\documentclass[11pt]{amsart}
\usepackage{graphicx}
\usepackage{amsmath}
\usepackage{amsthm,amsfonts,amssymb,mathrsfs,amscd,amstext,amsbsy}
\usepackage{epic,eepic}
\usepackage{yfonts}
\usepackage{paralist,enumerate}
\usepackage[all]{xy}
\usepackage{hyperref}
\hypersetup{colorlinks}

\newtheorem{theorem}{Theorem}[section]

\newtheorem{lem}[theorem]{Lemma}
\newtheorem{pro}[theorem]{Proposition}
\newtheorem{rem}[theorem]{Remark}
\newtheorem{exa}[theorem]{Example}

\newtheorem{theore}{Theorem}[section]

\newtheorem{corollar}{Corollary}

\addtocounter{corollar}{1}

\newtheorem{Definition}[theorem]{Definition}
\newtheorem*{Definition*}{Definition}

\newcommand\blfootnote[1]{%
  \begingroup
  \renewcommand\thefootnote{}\footnote{#1}%
  \addtocounter{footnote}{-1}%
  \endgroup }

\def\qed{\hfill \ifhmode\unskip\nobreak\fi\quad\ifmmode\Box\else$\Box$\fi\\ }

\begin{document}

\title[Stable subgroups and Morse subgroups in mapping class groups]{Stable subgroups and Morse subgroups in mapping class groups}
\author{Heejoung Kim}
\blfootnote{\textit {Date}: November 13, 2017.
\indent
\thanks{The author gratefully acknowledges the support of the NSF grant DMS-1405146.}}

\address{Department of Mathematics, University of Illinois at Urbana-Champaign, Urbana, IL 61801}
\email{hkim404@illinois.edu}

\begin{abstract}
For a finitely generated group, there are two recent generalizations of the notion of a quasiconvex subgroup of a word-hyperbolic group, namely a \textit{stable subgroup} and a \textit{Morse} or\textit{ strongly quasiconvex subgroup}. Durham and Taylor \cite{DT15} defined stability and proved stability is equivalent to convex cocompactness in mapping class groups. Another natural generalization of quasiconvexity is given by the notion of a Morse or strongly quasiconvex subgroup of a finitely generated group, studied recently by Tran \cite{T17} and Genevois \cite{G17}. In general, a subgroup is stable if and only if the subgroup is Morse and hyperbolic. In this paper, we prove that two properties of being Morse and stable coincide for a subgroup of infinite index in the mapping class group of an oriented, connected, finite type surface with negative Euler characteristic.
\end{abstract}

\maketitle

\section{Introduction}

\indent

The notion of a quasiconvex subgroup plays an important role in the theory of word-hyperbolic groups and in its various generalizations. For a group $G$ with a finite generating set $S$, a subgroup $H\leq G$ is \textit{quasiconvex} in $G$ with respect to $S$ if there is $N\geq 0$ such that every geodesic in the Cayley graph $\Gamma(G,S)$ of $G$ with respect to $S$ that connects a pair of points in $H$ is contained in the $N$-neighborhood of $H$.
In the context of word-hyperbolic groups, quasiconvexity is not dependent on a generating set of $G$, and for a subgroup being quasiconvex has several equivalent characterizations: being finitely generated and undistorted; being boundary quasiconvex-cocompact; being rational; and so on (see \cite{S91,GS91,KB02}). However, outside hyperbolic groups, quasiconvexity is not as useful since the notion depends on the choice of a generating set of the ambient group.

There have been several important recent generalizations of the notion of quasiconvexity to the context of subgroups of arbitrary finitely generated groups. In \cite{DT15}, Durham and Taylor introduced a strong notion of quasiconvexity for a subgroup of a finitely generated group, namely stability, which is preserved under quasi-isometry and implies hyperbolicity and quasiconvexity of the subgroup.

\begin{Definition}[Stability]\label{def1.1}
\textup{Let $G$ be a finitely generated group and let $H$ be a finitely generated subgroup of $G$. We say that $H$ is} stable\textup{ in $G$ if $H$ is undistorted and if for every (equivalently, some) finite generating set $S$ of $G$ and for every $k\geq1$ and $c\geq0$ there is some $L=L(S,k,c)$ such that for every pair of $(k,c)-$quasigeodesics in $G$ with the same endpoints on $H$, each qausigeodesic is contained in the $L$-neighborhood of the other.}
\end{Definition}

Stability has been characterized in many ways. For instance, in the recent studies of a generalization of the Gromov boundary, called the Morse boundary, Cordes and Durham in \cite{CD17} characterized stability in terms of ``boundary convex cocompactness''. See Section 4 in Cordes' survey paper \cite{C17} for details. We note that since a stable subgroup of an arbitrary group is hyperbolic, the geometry of the subgroup might not reflect the geometry of the whole group.

Another recent generalization of quasiconvexity is given by the notion of a Morse subset of a geodesic metric space. See \cite{CS15, C15, ACGH16, CH17, CD17, C17} for more details and additional related results.

\begin{Definition}[Morse subset]
\textup{Let $X$ be a geodesic meteoric space. A subset $Y\subseteq X$ is \textit{Morse} if every $k\geq1$ and $c\geq0$ there is some $M=M(k,c)$ such that every $(k,c)-$quasigeodesic in $X$ with endpoints on $Y$ is contained in the $M$-neighborhood of $Y$. }
\end{Definition}

This notion has been of particular interest in the case where $Y$ is a geodesic or quasigeodesic, where it has lead to the notion of a Morse boundary introduced by Cordes in \cite{C15}.
Moreover, Arzhantseva, Cashen, Gruber, and Hume in \cite{ACGH16} characterized Morse quasigeodesics in terms of superlinear divergence and sublinear contraction with characterizations of Morse sets.

 For the case where $X$ is the Cayley graph $\Gamma(G,S)$ of a finitely generated group $G$ with a finite generating set $S$, and $Y\subseteq X$ is a subgroup of $G$, the notion of a Morse subset naturally leads to the following:

\begin{Definition}[Morse or strongly quasiconvex subgroups]\label{def1.2}
\textup{Let $G$ be a finitely generated group and let $H$ be a subgroup of $G$. We say that $H$ is a \textit{Morse} or \textit{strongly quasiconvex} subgroup of $G$ if for every (equivalently, some) finite generating set $S$ of $G$, the subgroup $H\subseteq \Gamma(G,S)$ is Morse.   }
\end{Definition}

We note that if $H$ is a Morse subset in $\Gamma(G,S)$ for some finite generating set $S$ of $G$, then the same is true for every finite generating set of $G$. Moreover, $H$ is finitely generated and undistorted in $G$ in this case.

For a finitely generated group $G$, an element $g\in G$ of infinite order is called a \textit{Morse element} if the orbit of $\langle g \rangle$ in any Calyey graph is a Morse quasigeodesic. The notion of a Morse element has been studied for some time now. This case corresponds to cyclic Morse subgroups of finitely generated groups, and the cyclic Morse subgroups are indeed stable.

However, the actual notion of a Morse subgroup has not been defined until very recently. This notion was originally explicitly defined under the name ``strongly quasiconvex subgroup'' by Tran in July 2017 \cite{T17} with its characterization via lower relative divergence. Genevois independently introduced the term ``Morse subgroup'' for the notion from Definition \ref{def1.2} above in September 2017 \cite{T17}, with its characterizations in a cubulable group.

 We think that the term ``Morse subgroup'' may be preferable to ``strongly quasiconvex subgroup'' because of how the notion fits into the theory of Morse subsets and Morse boundary.

\vspace{0.3cm}
In the case of a hyperbolic group, every quasigeodesic stays in uniformly bounded distance to a geodesic by the Morse Lemma. Thus, both stability and being Morse are equivalent to quasiconvexity. Beyond hyperbolic groups, we have many examples to distinguish theses two definitions.

\begin{exa}\label{ex1}
\textup{The group $\mathbb{Z}^2$ is a non-stable Morse subgroup of $\mathbb{Z}^2$. Consider two $(3,0)$-quasigeodesics with endpoints $(0,0)$ and $(n,0)$ where $n\in \mathbb{N}$; one consists of three geodesics $[(0,0),(0,n)]$, $[(0,n),(n,n)]$, and $[(n,n),(n,0)]$, and the other consists of three geodesics $[(0,0),(0,-n)]$, $[(0,-n),(n,-n)]$, and $[(n,-n),(n,0)]$. Then if one contains the other in its $L$-neighborhood then $L>2n$. This means that there is no uniform bound $L=L(3,0)$ as described in Definition \ref{def1.1}. Therefore, the group $\mathbb{Z}^2$ is not stable but obviously Morse in $\mathbb{Z}^2$.
However, $\mathbb{Z}$ is not Morse in $\mathbb{Z}^2$. }
\end{exa}

Furthermore, non-hyperbolic peripheral subgroups of a relatively hyperbolic group and non-hyperbolic hyperbolically embedded subgroups of a finitely generated group are Morse but not hyperbolic, meaning that they are not stable (see \cite{DS05, S16}). Also, Tran gives examples of Morse subgroups which are not stable in right-angled Coxeter groups in \cite{T17}.

On the other hand, Tran proved (Theorem 1.16 in \cite{T17}) that for a non-trivial, infinite index subgroup of the right-angled Artin group $A_\Gamma$ of a simplicial, finite, connected graph $\Gamma$ which does not decompose as a nontrivial join, a Morse subgroup of $A_\Gamma$ is stable in $A_\Gamma$. This fact also follows from Theorem B.1 in Genevois' paper \cite{G17}. The results served as a motivation for the main result of this paper: the two notions coincide for a subgroup of infinite index in the mapping class group of an oriented, connected, finite type surface with negative Euler characteristic.

A surface $S$ is of finite type if $S$ is a compact surface minus a
 finite, possibly empty set of punctures, and we denote by $\chi(S)$ the Euler characteristic of $S$.

Our main theorem is the following:

\begin{theore}\label{thm1}
Let $S$ be an oriented, connected, finite type surface with $\chi(S)<0$ which is neither the 1-punctured torus nor the 4-punctured sphere.
Let \textup{Mod$(S)$} be the mapping class group of $S$, and let $G$ be a finitely generated subgroup of \textup{Mod$(S)$}. Then the following are equivalent:
\begin{enumerate}
\item $G$ is convex cocompact.
\item $G$ is Morse of infinite index in \textup{Mod$(S)$}.
 \end{enumerate}
\end{theore}

 For the implication $``(1)\Rightarrow (2)$'', we use the characterization of convex cocompactness in terms of stability proved in \cite{DT15}. For the reverse implication $``(2)\Rightarrow (1)$'', we use the characterization of convex cocompactness in terms of pseudo-Anosov elements proved in \cite{BBKL16}.
 Explicitly, we have the following corollary by combining their results and our main theorem. Furthermore, the next corollary says that stability and being Morse are equivalent notions in Mod$(S)$.

\begin{corollar}\label{co1}
Let $S$ be an oriented, connected, finite type surface with $\chi(S)<0$ which is neither the 1-punctured torus nor the 4-punctured sphere.
Let \textup{Mod$(S)$} be the mapping class group of $S$, and let $G$ be a finitely generated subgroup of \textup{Mod$(S)$}. Then the following are equivalent:
\begin{enumerate}
\item $G$ is convex cocompact.
\item $G$ is finitely generated, undistorted, and purely pseudo-Anosov.
\item $G$ is stable.
\item $G$ is Morse of infinite index in \textup{Mod$(S)$}.
 \end{enumerate}
\end{corollar}

The equivalence between (1) and (2) was proved in \cite{BBKL16} including the 1-punctured torus $S_{1,1}$ and the 4-punctured sphere $S_{0,4}$. On the other hand, the equivalence between (1) and (3) shown in \cite{DT15} excludes those two surfaces.

\begin{rem}
\textup{
The mapping class groups Mod$(S_{1,1})$ and Mod$(S_{0,4})$ are commensurable with $GL_2(\mathbb{Z})$ and thus are virtually free (see Chapter 2 in \cite{FM02}). Both these groups are locally quasiconvex and all of their finitely generated subgroups are stable and Morse. However, Mod$(S_{1,1})$ and Mod$(S_{0,4})$ contain Dehn twists and they have finitely generated subgroups that are not convex cocompact. Therefore, the conclusion of Theorem \ref{thm1} does not hold for these groups.
\\
\indent
The main result of \cite{DT15} and Corollary 1.2 in \cite{BBKL16} are stated slightly incorrectly since they should have omitted Mod$(S_{1,1})$ and Mod$(S_{0,4})$ for similar reasons.}
\end{rem}

Nevertheless, the equivalence between (3) and (4) includes the two surfaces $S_{1,1}$ and $S_{0,4}$ since stability is equivalent to being Morse in a hyperbolic group.

\vspace{0.3cm}

The paper is organized in the following way. In section \ref{s2}, we introduce the background and notions needed. In section \ref{s3}, we discuss some properties of Morse subgroups. Once this is completed, we prove two lemmas and Theorem \ref{thm1} in Section \ref{s4}.
\\
\\
\textbf{Acknowledgement.} The author would like to thank her PhD advisor Ilya Kapovich for his guidance, encouragement, and support. The author would also like to thank Chris
Leininger for helpful conversations about mapping class groups.

\section{Background}\label{s2}
In this section, we discuss some definitions and theorems for Theorem \ref{thm1}.
\subsection{Quasiconvexity}

\begin{Definition}
\textup{Let $(X, d_X)$ and $(Y,d_Y)$ be two metric spaces and let $f$ be a map from $X$ to $Y$. Then $f$ is a} $(K,C)$-quasi-isometric embedding \textup{if for every $x,y\in X$ we have }
\begin{center}
$\frac{d_X(x,y)}{K}-C\leq d_Y(f(x), f(y))\leq Kd_X(x,y)+C$.
\end{center}
\textup{The \textit{(K,C)}-quasi-isometric embedding $f$ is a} (K,C)-quasi-isometry \textup{if for every $y\in Y$ there is some $x\in X$ with $d_Y(f(x),y)\leq C$.}
\end{Definition}

For a map $i: X\to Y$ between metric spaces, we say that $X$ is \textit{undistorted}
in $Y$ if $i$ is a quasi-isometric embedding.

\begin{Definition}\label{def1}
\textup{Let $A$ and $X$ be two geodesic metric spaces and let $f: A\to X$ be a quasi-isometric embedding. We say $A$ is \textit{stable} in $X$ if for every $k\geq1$ and $c\geq0$ there is some $L=L(k,c)$ such that for every pair of $(k,c)-$quasigeodesic in $X$ with the same endpoints in $f(A)$, each quasigeodesic is contained in the $L$-neighborhood of the other. We say $A$ is \textit{Morse} or \textit{strongly quasiconvex} in $X$ if for every $k\geq1$ and $c\geq0$ there is some $M=M(k,c)$ such that every $(k,c)-$quasigeodesic in $X$ with endpoints on $f(A)$ is contained in the $M$-neighborhood of $f(A)$. }
\end{Definition}

We use those concepts for subgroups of a finitely generated group.
Let $G$ be a finitely generated group with a finite generating set $S$. Then we have a natural metric space with the associated word metric $d_S$, namely the Cayley graph $\Gamma(G,S)$. This has one vertex associated with each group element, and edges $(g,h)$ for $g,h\in G$ are assigned whenever $gh^{-1} \in S$. We note that the Cayley graph depends on the choice of a generating set.

We recall that for a finitely generated subgroup $H$ of $G$ and a finite generating set $T$ of $H$, the subgroup $H$ is called \textit{undistorted} if the inclusion map $(H,d_T)\to (G,d_S)$ is a quasi-isometric embedding.

Now we rewrite Definition \ref{def1.1} and \ref{def1.2} as follows.

\begin{Definition}
\textup{Let $G$ be a finitely generated group with a finite generating set $S$. Let $H$ be a finite generated subgroup of $G$. We say $H$ is \textit{stable} in $G$ if $H$ is undistorted in $G$ and $H\subseteq \Gamma(G,S)$ is stable for some (any) choice of a finite generating set of $H$.}
\end{Definition}

The notion of stability is independent of the choice of finite generating sets (see Section 3 in \cite{DT15}).

\begin{Definition}
\textup{Let $G$ be a finitely generated group with a finite generating set $S$, and let $H$ be a subgroup of $G$. We say that $H$ is \textit{Morse} or \textit{strongly quasiconvex} if for every $k\geq1$ and $c\geq0$ there is some $M=M(k,c)$ such that every $(k,c)-$quasigeodesic in $\Gamma(G,S)$ with endpoints on $H$ is contained in the $M$-neighborhood of $H$.}
\end{Definition}

 In fact, if $H$ is Morse in $\Gamma(G,S)$ for some finite generating set $S$ then $H$ is finitely generated and undistorted in $G$. Also, Morse subgroups are independent of the choice of finite generating sets (see Section 4 in \cite{T17}).

\subsection{Mapping class groups}

\indent
\vspace{0.2cm}

For the material in this section, we use \cite{FM02, KL08, FM10, DT15, BBKL16} as background references.

Let $S$ be an oriented, connected, finite type surface with negative Euler characteristic $\chi(S)$.
We note that a surface $S$ is of finite type if and only if the fundamental group of $S$ is finitely generated.
The \textit{(extended) mapping class group} Mod$(S)$ of $S$ is the group of isotopy classes of homeomorphisms of $S$. From the Nielsen-Thurston classification (see Section 13 in \cite{FM10}), for an element $f$ of Mod$(S)$, we have three possible cases; we say $f$ is \textit{periodic} if some power of $f$ is the identity; $f$ is \textit{reducible} if it permutes some finite collection of pairwise disjoint simple closed curves in $S$; and $f$ is \textit{Pseudo-Anosov} if it is neither periodic nor reducible. Also, if $f$ is reducible then some power of $f$ preserves a simple closed curve on $S$ up to isotopy (see \cite{FM10}).

In \cite{FM02}, Farb and Mosher introduced and developed the notion of a convex cocompact subgroup of the mapping class group of a closed, connected, and oriented surface by its action on Teichm$\ddot{\textup{u}}$ller space $\mathcal{T}(S)$. In \cite{KL08}, Kent and Leininger extended the definition for the case of an oriented, connected, finite area hyperbolic surface, which is equivalent to an oriented, connected, finite type surface with $\chi(S)<0$.

\begin{Definition}
\textup{Let $S$ be an oriented, connected, finite type surface with $\chi(S)<0$, and let Mod$(S)$ be its mapping class group. Then a subgroup $G<$ Mod$(S)$ is \textit{convex cocompact} if for some $x\in \mathcal{T}(S)$ the orbit $G \cdot x$ is quasiconvex with respect to the
Teichm$\ddot{\textup{u}}$ller metric on $\mathcal{T}(S)$.}
\end{Definition}

We remark that from the definition, if $G$ is convex cocompact, then $G$ is finitely generated, and every infinite order element in $G$ is pseudo-Anosov, i.e., $G$ is purely pseudo-Anosov. There have been many equivalent characterizations of convex cocompactness for subgroups of Mod$(S)$.

We first recall that there is a natural simplicial complex, associated to a surface $S$, called the \textit{curve complex} $\mathcal{C}(S)$ on which Mod($S$) acts by simplicial automorphisms. We restrict attention to one-skeleton of $\mathcal{C}(S)$ whose vertices are isotopy classes of essential simple closed curves on $S$, and two distinct isotopy classes are joined by an edge if they are disjointly realizable.

\begin{rem}
\textup{
For a surface $S$ which has either genus at least 2 or at least 5 punctures, the curve complex $\mathcal{C}(S)$ is connected.
The only surfaces making non empty curve complexes disconnected with $\chi(S)<0$ are the 1-punctured torus $S_{1,1}$ and the 4-punctured sphere $S_{0,4}$. In those two cases, the curve complex $\mathcal{C}(S)$ is a countable disjoint union of points. Hence, in the case of $S=S_{1,1}$ or $S=S_{0,4}$, we alter the definition of $\mathcal{C}(S)$ by joining a pair of distinct vertices if they realize the minimal possible geometric intersection in $S$, which makes $\mathcal{C}(S)$ connected.}
\end{rem}

In \cite{KL08}, Kent and Leininger show that a finitely generated subgroup $G$ of Mod($S$) is convex cocompact if and only if an orbit map from $G$ to the curve complex $\mathcal{C}(S)$ is undistorted. This fact is independently proved by Hamenst$\ddot{\textup{a}}$dt in \cite{H05}.
Furthermore, Durham and Taylor in \cite{DT15} characterize convex cocompactness in Mod$(S)$ by using only the geometry of Mod$(S)$ itself.

\begin{pro}[Theorem 1.1 in \cite{DT15}]\label{thm2}
Let $S$ be a connected and oriented surface which is neither the 1-punctured torus nor the 4-punctured sphere, and let \textup{Mod$(S)$} be its mapping class group. Then the subgroup $G<$ \textup{Mod$(S)$} is convex cocompact if and only if it is stable.
\end{pro}

In \cite{BBKL16}, we have another characterization for the convex cocompactness as follows.

\begin{pro}[Main Theorem in \cite{BBKL16}]\label{thm3}
Let $S$ be an oriented, connected, finite type surface with $\chi(S)<0$ and let \textup{Mod$(S)$} be its mapping class group. A subgroup $G<$ \textup{Mod$(S)$} is convex cocompact if and only if it is finitely generated, undistorted, and purely pseudo-Anosov.
\end{pro}

\section{Some properties of Morse subgroups}\label{s3}

In this section, we discuss some properties of Morse subgroups we need to prove Theorem \ref{thm1}. For the details, see Section 4 in \cite{T17}, where Morse subgroups are studied under the name strongly quasiconvex subgroups.
The following proposition tells us what the relation between Morse subgroups and stable subgroups is.

\begin{pro}[Proposition 4.3 of \cite{T17}]\label{pr0}
Let $G$ be a finitely generated group and let $H$ be an undistorted subgroup of $G$. Then $H$ is stable in $G$ if and only if $H$ is Morse and hyperbolic.
\end{pro}

The following proposition gives us a way to get another Morse subgroup from a Morse subgroup.

\begin{pro}[Theorem 4.11 in \cite{T17}]\label{pr1}
Let $G$ be a finitely generated group and let $A$ be an undistorted subgroup of $G$. If $H$ is a Morse subgroup of $G$, then $A\cap H$ is Morse in $A$.
\end{pro}

Now we recall some definitions for a subgroup of an arbitrary group.

\begin{Definition}
\textup{Let $G$ be a group and let $H$ be a subgroup of $G$. }

\begin{enumerate}
\normalfont{
 \item Conjugates $g_1Hg_1^{-1},\dots,$ $ g_kHg_k^{-1}$ are \textit{essentially distinct} if the cosets $g_1H,\dots,g_kH$ are distinct.
\item $H$ has \textit{finite height} $n$ in $G$ if the intersection of every $(n+1)$ essentially distinct conjugates is finite and $n$ is minimal possible.
\item $H$ has \textit{finite width} $n$ in $G$ if $n$ is the maximal cardinality of the set $\{g_iH :|g_iHg_i^{-1}\cap g_jHg_j^{-1}|=\infty \}$, where $\{g_iH \}$ ranges over all collections of distinct cosets.}

\end{enumerate}
\end{Definition}

 We note that every finite subgroup and every subgroup of finite index have finite height and width, and every infinite normal subgroup of infinite index has infinite height and width. The next proposition states that Morse subgroups are far from being normal.

\begin{pro}[Theorem 1.2 in \cite{T17}]\label{pr2}
Let $G$ be a finitely generated group and let $H$ be a Morse subgroup. Then $H$ has finite height and finite width.
\end{pro}

Motivated by Example \ref{ex1}, we obtain the following lemma. This is easily derived from Proposition \ref{pr2}.

\begin{lem}
Let $G$ be a finitely generated group and let $H$ be a Morse subgroup of $G$. For $g\in G$ and $h\in H$, if $g$ and $h$ commute and $\langle h, g\rangle \cong \mathbb{Z}^2$ is undistorted in $G$, then there is an integer $m>0$ such that $g^m$ is contained in $H$.
\end{lem}

\begin{proof}
Suppose that $g^m$ is not contained in $H$ for every $m>0$. Then we have $g^iH\neq g^jH$ if $i\neq j$. By Proposition \ref{pr2}, the subgroup $H$ has finite height $n$ in $G$ for some $n$. Since $h\in H$ and $g$ commute, we have $\langle h\rangle \in g^iHg^{-i}$ for $i=1,\dots, n+1$, and the intersection $\displaystyle\cap_{i=1}^{n+1}\, g^iHg^{-i}$ is infinite. This contradicts that $H$ has finite height $n$ in $G$. \end{proof}

\section{Proof of Theorem \ref{thm1}}\label{s4}

Throughout this section, we assume that $S$ is an oriented, connected, finite type surface with negative Euler characteristic which is neither the 1-punctured torus nor the 4-punctured sphere, and Mod$(S)$ is its mapping class group. For an essential simple closed curve $\alpha$, we denote by $D_\alpha$ the Dehn twist about $\alpha$, and denote by $C(D_\alpha)$ the centralizer.

The implication $``(1)\Rightarrow (2)$'' in Theorem \ref{thm1} is straight forward from Proposition \ref{thm2} and Proposition \ref{pr0}. In order to prove the reverse implication $``(2)\Rightarrow (1)$'', we need the following two lemmas.

\begin{lem}\label{lem1}
Let $\alpha$ be an essential simple closed curve on $S$. Then if $K$ is a Morse subgroup of $C(D_\alpha)$ then $K$ is either finite or has finite index in $C(D_\alpha)$.
\end{lem}

\begin{proof}
Suppose that $K$ is infinite and has infinite index in $C(D_\alpha)$. Since every element in $K$ commutes with $D_\alpha$, and $K$ has finite height in $C(D_\alpha)$ by Proposition \ref{pr2}, there exists $m>0$ such that $D_\alpha^m \in K$. Also, we have $[\,C(D_\alpha) : K\,]=\infty$ by the assumption. Therefore, there exists an infinite sequence $(g_n)$ of distinct elements in $C(D_\alpha)$ such that $g_iK\neq g_jK$ for $i\neq j$. Then from $D_\alpha^m \in K$ and $g_i\in C(D_\alpha)$, we have $\langle D_\alpha^m \rangle \subset g_iKg_i^{-1}$ for all $i$, and then the intersection $\displaystyle\cap_{i=1}^{\infty}\, g_iKg_i^{-1}$ is infinite. This contradicts that $K$ has finite height in $C(D_\alpha)$.
\end{proof}

\begin{lem}\label{lem2}
Let $\alpha$ be an essential simple closed curve on $S$. Then for each $g_0$ and $g$ in \textup{Mod$(S)$}, there exists a sequence of subgroups $g_0C(D_\alpha)g_0^{-1}=Q_0, Q_1\dots, Q_m= gC(D_\alpha)g^{-1}$ of \textup{Mod$(S)$} such that $Q_{i}\cap Q_{i+1}$ is infinite for each $i\in \{0,1,\dots,m-1\}$.
\end{lem}

\begin{proof}
Pick a path between $g_0\cdot\alpha$ to $g\cdot\alpha$ in the curve complex $\mathcal{C}(S)$ of $S$. For every vertex on the path, we pick a representative curve in the corresponding isotopy class of essential simple closed curves. Then there is a finite sequence of curves of $g_0\cdot\alpha=\gamma_0, \gamma_1, \dots , \gamma_m=g\cdot\alpha$ such that $\gamma_{i-1}$ is disjoint from $ \gamma_i$ for each $i$. Now we consider the sequence of subgroups $ g_0C(D_\alpha)g_0^{-1}$, $C(D_{\gamma_{1}}),\, \dots, C(D_{\gamma_{m-1}}),\, gC(D_\alpha)g^{-1}$. We have $\langle D_{\gamma_{i}}\rangle \subset C(D_{\gamma_{i}})\cap C(D_{\gamma_{i+1}})$ for $i\in \{1,2,\dots, m-2 \}$.
Since $g_0 D_\alpha g_0^{-1}=D_{g_{0}\cdot\alpha}$ and $gD_\alpha g^{-1}=D_{g\cdot\alpha}$, we have $\langle D_{\gamma_0}\rangle \subset g_0C(D_\alpha)g_0^{-1} \cap C(D_{\gamma_{1}})$ and $\langle D_{\gamma_m}\rangle \subset C(D_{\gamma_{m-1}})\cap gC(D_\alpha)g^{-1} $. This means the intersection of every two consecutive subgroups on the sequence is infinite.
\end{proof}

\begin{rem}
\textup{The centralizer of the Dehn twist about an essential simple closed curve is undistorted in Mod$(S)$ followed by Masur-Minsky distance formula in the marking graph (see \cite{MM99, BKMM12}). We note that the proof of Theorem \ref{thm1} is similar to the proof of Proposition 8.18 in \cite{T17}. In case of a right-angled Artin group $A_\Gamma$, Tran used a star subgroup to prove that stability is equivalent to being Morse. On the other hand, in Mod$(S)$, we can use the centralizer of a curve which plays the same role of the cone-off vertex in a star subgroup of $A_\Gamma$.}
\end{rem}
\vspace{0.2cm}

\begin{proof}[Proof of $``(2)\Rightarrow (1)$'' in Theorem \ref{thm1}]

Suppose that $G$ is Morse of infinite index but not convex cocompact. By Proposition \ref{thm3}, $G$ is not purely pseudo-Anosov. Without loss of generality, we may assume that there exists a reducible element $h\in G$ of infinite order such that $h$ fixes an essential simple closed curve $\alpha$ on $S$, i.e., $h(\alpha)=\alpha$. Let $D_\alpha$ be the Dehn twist about $\alpha$, and let $C(D_\alpha)$ be the centralizer of $D_\alpha$. Since $\langle h \rangle \subset G\cap C(D_\alpha)$, it is sufficient to show that $G\cap C(D_\alpha)$ is finite to derive a contradiction. In fact, we will show that for every $g\in\,$Mod$(S)$, $g^{-1}Gg\cap C(D_\alpha)$ is finite. Then in particular, if $g$ is the identity then $G\cap C(D_\alpha)$ is finite.

For a contradiction, suppose that there exists $g_0\,$$\in\,$Mod$(S)$ such that $H_0=g_0^{-1}Gg_0\cap C(D_\alpha)$ is infinite. Since $g_0^{-1}Gg_0$ is Morse and $C(D_\alpha)$ is undistorted in Mod$(S)$, $H_0$ is Morse in $C(D_\alpha)$ by Proposition \ref{pr1}. Take $K=H_0$ in Lemma \ref{lem1}, and then we have $[C(D_\alpha): H_0]$ is finite.

Now we claim that for every $g\in$Mod$(S)$, $g^{-1}Gg\cap C(D_\alpha)$ has finite index of $C(D_\alpha)$.
By Lemma \ref{lem2}, there is a sequence of subgroups $ g_0C(D_\alpha)g_0^{-1}=Q_0, Q_1,\dots, Q_m= gC(D_\alpha)g^{-1}$ such that $|Q_{i}\cap Q_{i+1}|=\infty$ for each $i\in \{0,1,\dots,m-1\}$.
Then since $[\,g_0C(D_\alpha)g_0^{-1}: G\cap g_0C(D_\alpha)g_0^{-1}\,]=[C(D_\alpha): H_0]<\infty$ and $|g_0C(D_\alpha)g_0^{-1} \cap Q_1|=\infty$, we have $G\,\cap \,Q_1$ is not finite.
Indeed, $Q_1$ is the centralizer of an essential simple closed curve in the proof of Lemma \ref{lem2}. Therefore, $Q_1$ is undistorted in Mod$(S)$ and $G\cap Q_1$ is Morse in $Q_1$. Then by Lemma \ref{lem1}, $G\cap Q_1$ has finite index in $Q_1$. By repeating this process, we end up with getting $[C(D_\alpha): g^{-1}Gg\cap C(D_\alpha)]=[gC(D_\alpha)g^{-1}: G\cap gC(D_\alpha)g^{-1}]=[Q_m : G\cap Q_m]<\infty$.

By Proposition \ref{pr2}, $G$ has finite height $k$ for some $k$ in Mod$(S)$. Since $[$Mod$(S)$$:G]=\infty$, there exist $k+1$ distinct elements $g_1, \dots,g_{k+1}$ of Mod$(S)$ such that $g_iG\neq g_j G$ for $i\neq j$. Then we have $[C(D_\alpha) : {g_i}^{-1}Gg_i\cap C(D_\alpha)]< \infty$ for all $g_i$ where $i=1,\dots,k+1$. It follows that $[C(D_\alpha) : (\displaystyle\cap_{i=1}^{k+1}\, {g_i}^{-1}Gg_i)\cap C(D_\alpha)]<\infty$. However, this means the intersection $\displaystyle\cap_{i=1}^{k+1}\, {g_i}^{-1}Gg_i$ is infinite, which contradicts that $G$ has finite height $k$ in Mod$(S)$.
Therefore, for any $g\in$Mod$(S)$, $g^{-1}Gg\cap C(D_\alpha)$ is finite. This completes the proof.
\end{proof}

\end{document}